\documentclass[a4paper,notitlepage, reqno]{amsart}
\usepackage{fullpage}
\usepackage[colorlinks=true,linkcolor=blue,unicode, psdextra]{hyperref}
\usepackage{color}
\usepackage{amsthm}
\usepackage{amsmath}
\usepackage{amssymb}
\usepackage{enumitem}
\usepackage{mathtools}
\usepackage{mathrsfs}
\usepackage[utf8]{inputenc}

\newtheorem*{theorem*}{Theorem}
\newtheorem*{proposition*}{Proposition}
\newtheorem*{lemma*}{Lemma}
\newtheorem*{corollary*}{Corollary}
\newtheorem*{remark*}{Remark}

\newtheorem{theorem}{Theorem}

\newtheorem{corollary}[theorem]{Corollary}
\newtheorem{proposition}[theorem]{Proposition}
\theoremstyle{definition}

\newtheoremstyle{assumption}% name
{\topsep}% space above
{\topsep}% space below
{}% body font
{}% indent amount
{\itshape}% theorem head font
{}% punctuation after theorem head
{.5em}% space after theorem head
{\thmname{#1}\thmnumber{ #2}\thmnote{ (#3)}}% theorem head spec
\theoremstyle{assumption}
\newtheorem{assumption}{Assumption}

\allowdisplaybreaks

\date{}
\title{A note on étale endomorphisms of normal schemes}
\author{Lázaro Orlando Rodríguez Díaz}
\address{Instituto de Matem\'{a}tica, 
	Universidade Federal do Rio de Janeiro, RJ, Brazil.}
\email{lazarord@im.ufrj.br}	

\begin{document}
	
\begin{abstract}
     We prove that under some extra hypothesis, given an étale endomorphism of a normal irreducible Noetherian and simply connected scheme, if the endomorphism is surjective then it is injective. The additional assumption concerns the possibility of constructing an étale cover out of a surjective étale morphism. We show that in some cases the surjectivity hypothesis can be removed if the intended étale cover is Galois.
\end{abstract}
\maketitle

\section*{}

A remarkable theorem due to Ax and Grothendieck states that an injective endomorphism of an algebraic variety is surjective, \cite[\S 14]{Ax68}, \cite{Ax69}, \cite[Theorem 10.4.11]{EGAIV_3}, \cite[Theorem 17.9.6]{EGAIV_4}. A straightforward converse to this theorem fails as surjective but not injective polynomials endomorphisms can be exhibited, e.g., $x\mapsto x^2$ in $\mathbb{C}$. A sort of converse to the Ax-Grothendieck theorem was  obtained in \cite{McKenna_Dries_90}. The present work arises from an attempt to prove a converse to the Ax-Grothendieck theorem in the case where the endomorphism we started with is étale and the scheme is normal. Even thought we do not succeeded in proving a fairly general converse under that hypothesis alone, we believe that our conditional results could serve as a starting point for further research in that direction.

We proceed to state two assumptions, all the results we obtained are conditioned to one of these. Let $S$ and $S'$ be normal irreducible Noetherian schemes, and let $f:S'\to S$ be a surjective étale morphism:

\begin{assumption}\label{assu1}
	There exists an étale cover $T\to S$ that factors into $T\to S' \to S$, where $T\to S'$ is surjective, and $T$ is an integral scheme.
\end{assumption}

\begin{assumption}\label{assu2}
   There exists a Galois cover $T\to S$ that factors into $T\to S' \to S$, where $T$ is an integral scheme.
\end{assumption}

\begin{proposition}\label{main_proposition}
	Let $S$ be a normal irreducible Noetherian and simply connected scheme. Let $f:S\to S$ be an étale endomorphism. If $f$ is surjective and satisfies the Assumption \ref{assu1}, then $f$ is injective.
\end{proposition}

The proof will follow from the hypothesis of simple connectivity as a consequence of our assumption.

\begin{proof}[Proof of Proposition \ref{main_proposition}]
Let $f: S'\to S$ be a surjective étale endomorphism of $S$, since we need to distinguish the domain from the codomain of $f$, here $S'=S$. By the Assumption \ref{assu1}, there exists an integral scheme $T$ and a finite étale morphism $h: T\to S$ such that $h=f\circ g$, where $g: T\to S'$ is surjective. Since $h$ is an étale cover, $T$ is connected and $S$ is simply connected, it follows  that $h$ is an isomorphism \cite[IV, Example 2.5.3]{Hartshorne77}. As $h$ is bijective, then $g$ is injective (hence bijective), and consequently $f$ is injective. 
\end{proof}

In this note $k$ denotes an algebraically closed field of characteristic zero and $\mathbb{A}^{n}_{k}=\operatorname{Spec}k[x_1,\dots,x_n]$ denotes the affine space over $k$. It is well known that $\mathbb{A}^{n}_{k}$ is simply connected, \cite[Theorem 2.9]{Wright81}. The following corollary is immediate from Proposition \ref{main_proposition}.

\begin{corollary}\label{cor_surj_implies_inj}
Let $f:\mathbb{A}^{n}_{k}\to \mathbb{A}^{n}_{k}$ be an étale endomorphism. If $f$ is surjective and satisfies the Assumption \ref{assu1}, then $f$ is injective.
\end{corollary}

Looking back to the proof of Proposition \ref{main_proposition}, we can reflect on the role played by the surjectivity hypothesis of the endomorphism $f: S\to S$. After all, we could restrict the codomain of the map to its image $f: S\to f(S)$, to turn the morphism surjective. The problem is that in general $f(S)$ is not simply connected. If we give up of the assumption of surjectivity and suppose $f$ is a separated morphism whose restriction satisfies the Assumption \ref{assu2}, we will prove that the induced extension of fields is Galois, as the following proposition shows.

\begin{proposition}\label{prop_function_fields}
Let $S$ be a normal irreducible Noetherian and simply connected scheme. Let $f:S'=S\to S$ be a separated étale endomorphism of $S$ such that $f:S'=S\to f(S)$ satisfies the Assumption \ref{assu2}. Then the induced field extension $K(S)\subset K(S')$ is a Galois extension. 
\end{proposition}

\begin{proof}[Proof of Proposition \ref{prop_function_fields}]
Let $f:S'=S\to S$ be a separated étale endomorphism of $S$. Since $f$ is étale it is flat, then it is an open map \cite[III, Exercise 9.1]{Hartshorne77}, therefore $f(S)$ is an open subscheme of $S$, in particular, $f(S)$ is a normal irreducible Noetherian scheme and $K(S)=K(f(S))$.  Let's restrict the codomain of $f$ to its image and rename the map $f$ as $\tilde{f}:S'\to f(S)$. The morphism $\tilde{f}$ is étale and surjective, then since $\tilde{f}$ satisfies the Assumption \ref{assu2} there exists an integral scheme $T$ and a finite étale morphism $h: T\to f(S)$ such that $h=\tilde{f}\circ g$, where $g: T\to S'$, moreover, since $h$ is a Galois cover we have that $K(f(S))\subset K(T)$ is a Galois extension. As $h=\tilde{f}\circ g$ is étale and $\tilde{f}$ is étale, it follows that $g$ is étale \cite[Proposition 17.3.4]{EGAIV_4}. Note that $\tilde{f}$ is separated because $f$ is separated \cite[II, Corollary 4.6, f) ]{Hartshorne77}. On the other hand, as $h=\tilde{f}\circ g$ is finite and $\tilde{f}$ separated, it follows that $g$ is finite \cite[Lemma 5.3.2, 1)]{Szamuely09}. Then we have that $g$ is finite étale, i.e., $g: T\to S'$ is an étale cover. We know that $S'$ is simply connected and $T$ is connected, then $g$ is an isomorphism, in particular $K(S')=K(T)$. Since $K(f(S))\subset K(T)$ is a Galois extension, we conclude that $K(S)\subset K(S')$ is a Galois extension as well.
\end{proof} 

The above proposition is particularly useful if we recall the following theorem proved by Campbell for $k=\mathbb{C}$ \cite{Campbell73}; for an arbitrary algebraically closed field of characteristic zero the proof can be found in \cite[Theorem 21.11]{Abhyankar77}, \cite[Theorem 2]{Razar79} and \cite[Theorem 3.7]{Wright81}.

\begin{theorem}{\cite{Campbell73}}\label{Campbell_theo}
Let $S'=\mathbb{A}^{n}_{k}\to S=\mathbb{A}^{n}_{k}$ be an étale endomorphism. Then $f$ is an isomorphism if and only if the induced field extension $K(S)\subset K(S')$ is a Galois extension. 
\end{theorem}

The following corollary is a direct consequence of Proposition \ref{prop_function_fields} and Campbell's Theorem \ref{Campbell_theo}, since any morphism of affine schemes is separated \cite[II, Proposition 4.1]{Hartshorne77}. It shows that the surjectivity hypothesis in Corollary \ref{cor_surj_implies_inj} could be removed. 

\begin{corollary}
Let $f:\mathbb{A}^{n}_{k}\to \mathbb{A}^{n}_{k}$ be an étale endomorphism such that $f:\mathbb{A}^{n}_{k}\to f(\mathbb{A}^{n}_{k})$ satisfies the Assumption \ref{assu2}. Then $f$ is an isomorphism.
\end{corollary}

\bibliographystyle{abbrv}
\bibliography{bibliografia}

\end{document}